\documentclass[final,5p,times,authoryear]{elsarticle}

\usepackage{amsthm}
\usepackage{graphicx}
\usepackage{pgfplots}
\usepackage{subcaption}
\usepackage{cite}
\usepackage{amsmath,amssymb,amsfonts}
\usepackage{mathrsfs}
\usepackage{algorithmicx}
\usepackage{textcomp}
\usepackage{xcolor}
\usepackage{booktabs}
\usepackage{array}
\usepackage[ruled,lined]{algorithm2e}
\usepackage{multirow}
\usepackage{upgreek}
\usepackage{bbold}
\usepackage{tikz}
\usepackage{tikz-network}
\usepackage{todonotes}
\usepackage{float} 
\usepackage{xcolor}

\newif\ifshowcolor
\showcolorfalse

\newcommand{\mycolor}[2][blue]{%
  \ifshowcolor
    \textcolor{#1}{#2}%
  \else
    #2%
  \fi
}

\newtheorem{theorem}{Theorem}
\newtheorem{problem}{Problem}
\newtheorem{proposition}{Proposition}

\newtheorem{definition}{Definition}

\newtheorem{lemma}{Lemma}

\DeclareMathOperator{\proj}{Proj}

\journal{European Journal of Control}

\begin{document}

\begin{frontmatter}


\title{A Passivity Analysis for Nonlinear Consensus on Balanced Digraphs\tnoteref{label1}}
\tnotetext[label1]{This work was supported by the Israel Science Foundation grant no.\,453/24 and the Gordon Center for Systems Engineering.}

\author{Feng-Yu Yue\corref{cor1}\fnref{label2}}
\ead{fengyu.yue@campus.technion.ac.il}
\author{Daniel Zelazo\fnref{label2}}
\ead{dzelazo@technion.ac.il}

\cortext[cor1]{Corresponding authors.}
\affiliation[label2]{organization={Faculty of Aerospace Engineering},
           addressline={Technion
– Israel Institute of Technology}, 
           city={Haifa},
           postcode={3200003}, 
           country={Israel}}

%% Abstract
\begin{abstract}
This work deals with the output consensus problems for multi-agent systems over balanced digraphs. \mycolor{While passivity-based approaches are widely used for analyzing undirected consensus protocols, we show that they are generally not applicable to the directed linear consensus protocol.} To address this limitation, we propose a general approach that enables a passivity-based analysis for network systems with directed couplings. Then, we mitigate the complexity introduced by nonlinearities and directed interconnections by reformulating the general output consensus problem as a convergence analysis on a submanifold. \mycolor{Within this framework, we further focus on the stabilization problem, a specific form of the output consensus problem,} and establish a sufficient passivity-based condition for stabilizing multi-agent systems over balanced digraphs. The results are supported by a numerical example.
\end{abstract}


% \begin{highlights}
% \end{highlights}

%% Keywords

\end{frontmatter}



\section{Introduction}
Multi-agent systems (MASs) have received extensive attention in both industrial practice and theoretical research, ranging from smart grids, distributed sensing and transportation networks to control, robotics and computer science
\citep{bullo2020lectures, BearingOnly_TAC2015, AerialSwarm_TRO2018,ReinforcementLearning}. From a control perspective, a fundamental challenge in this field is the consensus problem \citep{scardovi2009sync_bionet}, which aims to coordinate agent dynamics to achieve agreement on a shared state or trajectory.

Consensus analysis requires understanding the fundamental interplay between agent dynamics, information exchange structures, and interaction protocols in MASs \citep{Burger2013AUTO_duality,Sharf2019MIMO}. Diffusively-coupled networks provide a canonical architecture for studying these relationships \citep{mesbahi2010graph}. Their inherent structure, composed of symmetric and feedback interconnections, makes passivity theory a natural tool for its analysis \citep{bai2011cooperative}. \mycolor{Passivity theory can simplify the study of complex network systems, as it enables a decoupled treatment of network dynamics and network topology, and is inherently linked to stability and convergence \citep{Burger2013AUTO_duality}.} Arcak's seminal work \citep{arcak2007passivitydesign} leveraged passivity to characterize network convergence behavior. This approach was later extended into a comprehensive passivity-based cooperative control framework for single-input single-output (SISO) systems \citep{Burger2013AUTO_duality}. 
This framework revealed a connection between the network system steady-states and dual network optimization problems \citep{Rockafellar1998NetOpt}. 

While the passivity framework has proved very powerful, it relies heavily on the symmetric feedback interconnection of the incidence matrix in diffusively coupled networks. This symmetry requirement confines the framework to systems with undirected interconnections. Replacing one of the incidence matrices in the structure (detailed in Section \ref{sec:preliminary}) enables the representation of directed graph topologies but sacrifices the diffusive coupling property due to the loss of symmetry. Moreover, given passive edge controllers, the feedback path in the loop may not preserve passivity for the entire interconnection. These challenges hinder passivity-based analysis for MASs over digraphs, motivating a generalized approach that extends passivity theory’s benefits to solving the consensus problem of directed topologies.

On the other hand, since \mycolor{passivity theory} enables a separate analysis of system dynamics and the underlying graphs, consensus problems can be categorized by the linearity of the system dynamics and the graph directionality. The problems can be classified, by increasing complexity, as linear dynamics over undirected graphs (e.g., the standard linear consensus protocol), nonlinear dynamics over undirected graphs (e.g., \citep{Burger2013AUTO_duality}), linear dynamics over digraphs (e.g., the linear consensus protocol for digraphs), and nonlinear dynamics over digraphs (e.g., \citep{li2019passivition,li2020jointlypassivation}). Also, consensus behaviors manifest in two distinct forms: \emph{average consensus}, where agent states converge to the mean of initial conditions, and \emph{regular consensus}, where states agree on the same value (not necessarily the average). When applying linear consensus protocols, systems over connected undirected graphs achieve average consensus, whereas systems over digraphs containing a globally reachable node only achieve regular consensus. However, systems with balanced digraphs restore average consensus capability, suggesting balanced digraphs' unique intermediate position between directed and undirected topologies. This motivates the investigation into balanced digraphs in this paper.

This paper focuses on a \mycolor{challenging} consensus problem within the above taxonomy: nonlinear dynamics over digraphs. The works \citep{li2019passivition,li2020jointlypassivation} were related to this topic and developed a passivation approach, but they only considered the case where the controllers are linear static maps and didn't provide a general analysis method for network systems with directed coupling. Montenbruck et al. \citep{Max2017submanifold} regarded the agreement space as a submanifold and developed powerful analytical tools to establish connections between passivity properties and stabilization around a submanifold, yielding explicit controller synthesis methods. However, they considered all controllers as a single entity without examining the passivity of each individual agent or controller, so it did not allow for an in-depth investigation of the interplay among the controller dynamics, the agent dynamics, and the underlying digraphs within the MASs.

In this paper, we conduct a passivity analysis for MASs interconnected via balanced digraphs and investigate the relationship between passivity, output consensus and \mycolor{stabilization (a specific form of output consensus) in such systems}. Our contributions are as follows. We begin by discussing the difference between the diffusively coupled network and its variant for digraphs. Our analysis uncovers a potential loss of passivity in the feedback path of the variant structure for general digraphs, even under the fundamental linear consensus protocol. Building on these insights, we develop a generalized approach that enables a passivity-based analysis for systems with directed couplings. \mycolor{By reformulating the output agreement problem as convergence to a submanifold, we derive passivity conditions for agents and controllers that stabilize network systems over balanced digraphs.}

The remainder of the paper is organized as follows. Section \ref{sec:preliminary} introduces preliminaries on \mycolor{networked systems and submanifold stabilization}.
Section \ref{sec:general_structure} proposes a general approach for analyzing directed coupling and reformulates the output agreement problem. Section \ref{sec:main_result} presents a passivity-based analysis for stabilizing networked systems interconnected via balanced digraphs. Numerical examples and concluding remarks are given in Sections \ref{sec:case} and  \ref{sec:conclusion}.

\paragraph*{Notations}
The notation $\mathbb{1}_n$ ($\mathbb{0}_n$) denotes the $n$-dimensional vector of all ones (zeros), and $I_n$ represents the $n\times n$ identity matrix, where the subscript $n$ may be omitted when the dimension is clear from the context. For a set $A$, its cardinality is denoted by $|A|$. We denote the kernel of a linear transformation $T:\ X\to Y$ by $\ker(T)$, and the orthogonal complement of a subspace $U$ by $U^\perp$. %, and the orthogonal projection of some $x\in\mathbb{R}^n$ onto $U$ by $\proj_U(x)$.

Fundamental notions from algebraic graph theory are also used in this paper. A directed graph $\mathcal{G}=(\mathbb{V}, \mathbb{E})$ comprises of a finite vertex set $\mathbb{V}$ and an edge set $\mathbb{E}\subset\mathbb{V}\times \mathbb{V}$. The incidence matrix $E\in \mathbb{R}^{|\mathbb{V}|\times|\mathbb{E}|}$ is defined as follows. $[E]_{ik}:=1$ if $i$ is the head of edge $e_k$, $[E]_{ik}:=-1$ if $i$ is the tail of edge $e_k$ and $[E]_{ik}:=0$ otherwise.
We decompose the incidence matrix into the out-incidence matrix $B_o$ and the in-incidence matrix $B_i$ \citep{restrepo2021edgeLyapunov}, i.e., $E=B_o+B_i$, where: $[B_o]_{ik}:=1$ if $i$ is the head of edge $e_k=(i,k)$ and $[B_o]_{ik}:=0$ otherwise; $[B_i]_{ik}:=-1$ if $i$ is the tail of edge $e_k$ and $[B_o]_{ik}:=0$ otherwise. The graph Laplacian of undirected graphs is defined as $L=EE^\top$. For digraphs, we define the in-Laplacian matrix $L_i=-B_iE^\top$ and out-Laplacian matrix $L_o=B_oE^\top$.

\section{Preliminaries}\label{sec:preliminary}
This section presents \mycolor{two key structures of network systems}, diffusively-coupled networks and their directed variants, along with essential passivity concepts. We then outline the mathematical framework for reformulating output consensus as a convergence analysis on a submanifold.
\subsection{Balanced digraphs and globally reachable nodes}\label{sec:balanced}
A digraph is called \emph{balanced} if the in-degree equals the out-degree for every node. 
\begin{lemma}\label{lemma:gmatrix}
    For incidence matrix $E\in \mathbb{R}^{|\mathbb{V}|\times |\mathbb{E}|}$ and out-Laplacian $L_o\in \mathbb{R}^{|\mathbb{V}|\times |\mathbb{V}|}$, the following statements are equivalent:
\begin{enumerate}
    \item[$i$)] The digraph is balanced,
    \item[$ii$)] $E^\top \mathbb{1}_{|\mathbb{V}|}=\mathbb{0}_{|\mathbb{E}|}$, and $E\mathbb{1}_{|\mathbb{E}|}=\mathbb{0}_{|\mathbb{V}|}$,
    \item[$iii$)] $L_o\mathbb{1}_{|\mathbb{V}|}=\mathbb{0}_{|\mathbb{V}|} \text{ and} \ L_o^\top\mathbb{1}_{|\mathbb{V}|}=\mathbb{0}_{|\mathbb{V}|} $.
\end{enumerate}
\end{lemma}
\begin{proof}
    $i$) $\Leftrightarrow$ $iii$): We recommend readers to refer to Lemma 6.4 in \citep{bullo2020lectures}.

    $i$) $\Leftrightarrow$ $ii$):
    It is sufficient to show the equivalence between statement $i$) and $E\mathbb{1}_{|\mathbb{E}|}=\mathbb{0}_{|\mathbb{V}|}$. The incidence matrix can be represented as $E=B_o+B_i$. For row $i$ ($i=1,\ldots,|\mathbb{V}|$) of $B_i$($B_o$), the row-sum of row $i$ is the in(out)-degree of the corresponding node $i$. Thus, the given digraph is balanced if and only if $E\mathbb{1}_{|\mathbb{E}|}=B_o{1}_{|\mathbb{E}|}+B_i{1}_{|\mathbb{E}|}=\mathbb{0}_{|\mathbb{V}|}$.
\end{proof}

In this work, we will use the notion of globally reachable nodes and walks. \mycolor{A \emph{walk} in a directed graph is a sequence of nodes connected by directed edges that point from one node to the next in the sequence.}
\begin{definition}[\mycolor{\citep{bullo2020lectures}}]
    A directed graph possesses a \emph{globally reachable node} if one of its nodes can be reached from any other node by traversing a directed walk.
\end{definition}

\subsection{Network systems and passivity}
Consider a population of agents interacting over a network $\mathcal{G}=(\mathbb{V},\mathbb{E})$, where the vertices $\mathbb{V}$ denote the set of agents and the edges $\mathbb{E}$ represent edge controllers describing the interaction between agents. Each agent $\{\Sigma_i\}_{i\in\mathbb V}$  and controller $\{\Pi_k\}_{k\in\mathbb E}$ are described by the SISO nonlinear dynamical systems,
\begin{equation}\label{eq:agents_decomp}
        \Sigma_i:\begin{cases}
            \dot{x}_i(t)=f_i(x_i(t),u_i(t)), \\
        y_i(t)=h_i(x_i(t),u_i(t)),
        \end{cases}
\end{equation}
\begin{equation}\label{eq:controllers_decomp}
    \Pi_k: \begin{cases}
        \dot{\eta}_k(t)=\phi_k(\eta_k(t),\zeta_k(t))\\
        \mu_k(t)=\psi_k(\eta_k(t),\zeta_k(t))
    \end{cases}
\end{equation}
\mycolor{Note that agents and controllers are interconnected in parallel, respectively.} Define the stacked inputs of agents $ {u}(t)=[u_1, \cdots, u_{|\mathbb{V}|}]^\top$, and similarly for outputs of agents $ {y}(t)$, inputs of controllers $ {\zeta}(t)$ and outputs of controllers $ {\mu}(t)$. 

\begin{figure*}[t]
    \centering
    \tikzstyle{block} = [draw, rectangle, 
    minimum height=2.5em, minimum width=4em]
\tikzstyle{sum} = [draw, circle, node distance=1cm]
\tikzstyle{input} = [coordinate]
\tikzstyle{output} = [coordinate]
\tikzstyle{pinstyle} = [pin edge={to-,thin,black}]
\begin{subfigure}{0.45\textwidth}
\resizebox{0.9\textwidth}{!}{
\begin{tikzpicture}[auto, node distance=2.5cm,>=latex', scale=.5]
    % We start by placing the blocks
    \node [input, name=input] {};
    \node [sum, right of=input, node distance=0.8cm] (sum) {};
    \node [block, below of=sum,
            node distance=1.5cm] (Eout) {
                $N$};
    \node [block, right of=sum] (agents) {$\Sigma$};
    \node [input, name=center, below of=agents] {};
    \node [block, below of=agents,
            node distance=3cm] (controllers) {$\Pi$};
    \node [output, right of=agents] (output) {};
    \node [block, below of=output,
            node distance=1.5cm] (ETout) {$M$};

    \draw [draw,->] (input) -- node {$p(t)$} (sum);
    \draw [->] (sum) -- node {$u(t)$} (agents);
    \draw [->] (agents) -| node [near start] {$y(t)$}  (ETout);
    \draw [->] (ETout) |- node [near end] {$\zeta(t)$}  (controllers);
    \draw [->] (controllers) -| node [near start] {$\mu(t)$}  (Eout);
    \draw [->] (Eout) -- node {$g(t)$} node[pos=0.95] {$-$}  (sum);
\end{tikzpicture} }   
    \caption[]{}
    \label{fig:DC_directed}
    \end{subfigure}
    \begin{subfigure}{0.45\textwidth}
    \centering
\resizebox{\textwidth}{!}{
\begin{tikzpicture}[auto, >=latex', scale=0.45]
    % First add the background library
    \usetikzlibrary{fit,backgrounds}
    
    % We start by placing the blocks
    % Nodes placement
    \node [input, name=winput] {};
    \node [sum, right of=winput,node distance=1.5cm] (sum) {};
    \node [block, below of=sum, node distance=1.25cm] (E) {$E$};
    \node [block, right of=sum, node distance=2cm] (Sigma) {$\Sigma$};
    \node [output, right of=Sigma, node distance=2cm] (output) {};
    \node [output, right of=output, node distance=2cm] (output1) {};
    \node [output, right of=output1, node distance=0.8cm] (output2) {};
    
    % Second set of blocks on the lower side
    \node [block, below of=Sigma, node distance=2.5cm] (Pi1) {$\Pi$};
    \node [block, below of=output, node distance=1.25cm] (ET1) {$E^\top$};
    
    % Block on the left side
    \node [block, below of=output1, node distance=2cm] (ET2) {$E^\top$};
    \node [block, left of=ET2, node distance=8cm] (Bi) {$B_i$};
    \node [block, below of=Pi1, node distance=1.5cm] (Pi2) {$\Pi$};
    
    % Add the background shaded box for inner loop
    \begin{scope}[on background layer]
        \node[fill=gray!10, rounded corners, draw=gray!20, inner sep=8pt, fit=(sum) (E) (Sigma) (ET1) (Pi1)] {};
    \end{scope}
    
    % Connections
    \draw [->] (sum) -- node {$u(t)$} (Sigma);
    \draw [->] (output) -- (output2);
    \draw [->] (Sigma) -| node [near start] {$y(t)$} (ET1);
    \draw [->] (output1) -- (ET2);
    \draw [->] (ET1) |- node [near end] {$\zeta(t)$}  (Pi1);
    \draw [->] (Pi1) -| node [near start] {$\mu(t)$}  (E);
    \draw [->] (E) -- node {$z(t)$} node[pos=0.95] {$-$}  (sum);
    \draw [->] (ET2) |- node [near end] {$\zeta(t)$}  (Pi2);
    \draw [->] (Pi2) -| node [near start] {$\mu(t)$}  (Bi);
    \draw [->] (Bi) |- node [pos=0.6] {$w(t)$} node[pos=0.95] {$+$} (sum);
\end{tikzpicture}}
    \caption[]{}
    \label{fig:decomp}
\end{subfigure}
\caption{(a) Block-diagram of networked systems where matrices $M$ and $N$ represent the inter-agent interconnection. For $(\Sigma, \Pi, \mathcal G)_{E}$ we set $M=N^\top$ and the matrix $N=E$; for $(\Sigma, \Pi, \mathcal G)_{B_o}$ we set $N=B_o$ and $M=E^\top$. 
(b) A loop decomposition for the system $(\Sigma, \Pi, \mathcal G)_{B_o}$. We denote this structure by $(\Sigma,\Pi,\mathcal{G},w)$.}
\end{figure*}
\mycolor{As shown in Figure \ref{fig:DC_directed}, a network system can be represented by agent dynamics $\Sigma$ and controller dynamics $\Pi$, and two matrices $M$ and $N$ encoding the interconnection of the system. The feedback equations are then completed with the algebraic relations, $\zeta(t) = My(t)$ and $u(t) = p(t)-N\mu(t)$, where $p(t)$ denotes the exogenous input. In this paper, we only consider the case where $p(t)=0$.}

\mycolor{The matrices $M$ and $N$ can be selected to represent different information exchange topologies.} A network system described by Figure \ref{fig:DC_directed} is called \emph{diffusively coupled} \mycolor{if $M=N^\top$ and the matrix $N$ is set to the incidence matrix $E$}. In this configuration, the topology of the system is characterized by the undirected counterpart of $\mathcal{G}$, where the edges between the agents represent bidirectional communication links.

\mycolor{When matrices $N$ and $M$ are set to the out-incidence matrix $B_o$ and the transpose of the incidence matrix $E^\top$, respectively}, the resulting structure represents the networked system interconnected by digraphs. In this case, the system can no longer be considered diffusively coupled, as the inherent symmetry of the underlying graph is broken. We adopt the notation $(\Sigma,\Pi,\mathcal{G})_E$ for diffusively-coupled networks and denote the directed case by $(\Sigma,\Pi,\mathcal{G})_{B_{o}}$.

Now, let's take the system (\ref{eq:agents_decomp}) as an example to introduce the definition of passivity \citep{Khalil2008Nonlinear}.
\begin{definition}
    For the SISO system \eqref{eq:agents_decomp}, if there exists a positive semi-definite storage function $V_i(x_i)$ and scalers $\varepsilon_i$ and $\delta_i$ such that,
    \begin{equation}
        \dot{V_i}(x_i)\leq u_iy_i-\varepsilon_iy_i^2-\delta_i u_i^2, \quad \forall x_i,u_i,y_i,
    \end{equation}
    then, the system \eqref{eq:agents_decomp} is said to be
    \begin{enumerate}
        \item \emph{passive} if $\varepsilon_i=0$ and $\delta_i=0$,
        \item \emph{output strictly passive} (OP-$\varepsilon_i$)  if $\varepsilon_i>0$ and $\delta_i\geq0$,
        \item \emph{input strictly passive} (IP-$\delta_i$) if $\varepsilon_i\geq0$ and $\delta_i>0$.
    \end{enumerate}
\end{definition}%
\noindent The maximal $\delta_i$ and $\varepsilon_i$ are the \emph{passivity indices} of the system. 

Given passive edge controllers, the feedback path of $(\Sigma,\Pi,\mathcal{G})_E$ preserves the passivity while that of $(\Sigma,\Pi,\mathcal{G})_{B_o}$ may lose passivity. We will discuss the differences between the two structures from a passivity perspective in Section \ref{sec:general_structure}.

\subsection{Tools for analyzing convergence to a submanifold}
This paper performs a passivity-based analysis for the output consensus problems of network systems governed by balanced digraphs. Let $S$, $S^\perp$, and the vector $y(t)$ denote the \emph{agreement space} ${\rm span}(\mathbb{1})$, \emph{disagreement space}, and the output of a system at time $t$, respectively. The system is said to achieve \emph{asymptotic output agreement} if the output satisfies,
% \begin{equation}
    $$\lim_{t\to\infty}y(t)=c\mathbb{1}\in S,$$
% % \end{equation}
where $c\in\mathbb{R}$ is called the agreement value. On the other hand, $S$ is a submanifold, so the above definition implies that $y(t)$ \emph{asymptotically converges to the agreement submanifold $S$}. This allows for converting the output agreement problem into a study of the convergence on the agreement submanifold.

\mycolor{In this paper, we will consider signals in $\mathscr{L}^2$ space, as defined by,
\begin{equation}\label{eq:space_Lp}
    \mathscr{L}^2=\left\{f:\mathbb{R}\to \mathbb{R}^n \ | \ f \ \text{ measurable}, \left(\|f\|_{\mathscr{L}^2}\right)^{2}<\infty\right\}, 
\end{equation}
where 
$$\|f\|_{\mathscr{L}^2}=\left(\int_\mathbb{R} \|f(t)\|^2 {\rm d}t\right)^{1/2}$$
denotes the $\mathscr{L}^2$ norm of signal $f$, "measurable" means Lebesgue measurable, and ${\rm d}t$ is short for ${\rm d}\lambda(t)$, with $\lambda$ being the Legesgue measure on $\mathbb{R}$\citep{Max2017submanifold}. It is important to highlight that, with the inner product
\begin{equation}
    \langle\cdot,\cdot\rangle: \mathscr{L}^2\times\mathscr{L}^2\to\mathbb{R}, \langle x,y\rangle\mapsto\int_{-\infty}^\infty x^\top(t)y(t){\rm d}t,
\end{equation}
$\left(\mathscr{L}^2,\langle\cdot,\cdot\rangle\right)$ is a Hilbert space, where we can define passivity.}

\mycolor{As noted in \citep[Chapter 5]{Khalil2008Nonlinear}, mappings from $\mathscr{L}^2$ to $\mathscr{L}^2$ are insufficient for describing unstable systems. To overcome this limitation, we introduce its extended space, $\bar{\mathscr{L}}^2$:
\begin{equation}
    \bar{\mathscr{L}}^p= \left\{f\ | \ f^\tau\in\mathscr{L}^p, \  \forall \tau\in[0,\infty]\right\},
\end{equation}
where $f^\tau(t)=f(t)$ if $0\leq t\leq \tau$ and $f^\tau(t)=0$ if $t>\tau$.}
% Note that $f^\tau$ is called a truncation of $f$.

On the other hand, to analyze the convergence to the agreement submanifold, it is necessary to define a new space with respect to it. Montenbrunk et al. \citep{Max2017submanifold} introduced a suitable space $\mathscr{L}_M^p$ and its extended space $\bar{\mathscr{L}}_M^p$, which can be effectively employed in our analysis. \mycolor{For the output agreement problem, set $M=S$ and $p=2$.} Since \mycolor{$S$} is a smoothly embedded submanifold, it has a tubular neighborhood $U$ by the \emph{tubular neighborhood theorem} \citep[Chapter 10]{lee2024smoothmanifolds}.
Then, we can define the space,
{\small 
    $$\mycolor{\mathscr{L}_S^2}=\left\{f:\mathbb{R}\to U|f \, \text{measurable}, \int_\mathbb{R}d(f(t),S)^2 {\rm d}t<\infty\right\},$$
}%
where \mycolor{$d(f(t),S)$ denotes the infimal Euclidean distance from all the points in $S$ to $f(t)$.} \mycolor{Similarly, define the truncation and extended space $\bar{\mathscr{L}}_S^2$. The truncation should map any signal to the desired submanifold $M$, so it can be chosen as the orthogonal projection onto $S$,} $$r:\mathbb{R}^n\to S, \ x\mapsto\proj_S(x)=\tfrac{1}{n}\mathbb{1}_n\mathbb{1}_n^\top(x).$$ 
Then, the extended space $\bar{\mathscr{L}}_S^2$ is defined by
\begin{equation}
    \bar{\mathscr{L}}_S^2= \left\{f|f_S^\tau\in\mathscr{L}_S^2, \  \forall \tau\in[0,\infty]\right\},
\end{equation}
where
\begin{equation}\label{eq:f_S_tau}
    f_S^\tau(t)=\begin{cases}
        f(t), \ & \text{if} \ 0\leq t\leq \tau ,\\
        r(f(t)), \ & \text{otherwise.}
    \end{cases}
\end{equation}
% $f_S^\tau(t)=f(t)$ when $t\in[0, \tau]$ and $f_S^\tau(t)=r(f(t))$ otherwise.

\mycolor{When applying the extended space $\bar{\mathscr{L}}_S^2$, we encounter a limitation. While we can define a mapping $\|\cdot\|_{\mathscr{L}_S^2}: \mathscr{L}_S^2\to \mathbb{R}, \ f\mapsto \int_\mathbb{R}d(f(t),S)^2{\rm d}t$, this fails to satisfy the triangle inequality and thus isn't a norm.} To address this, we introduce a mapping $\Theta_S$,
\begin{equation}\label{eq:ThetaS}
     \Theta_S: \ \bar{\mathscr{L}}_S^2\to \bar{\mathscr{L}}^2, \quad f(t)\mapsto \proj_{S^\perp}(f(t)),
\end{equation}
where $\proj_{S^\perp}(f(t))=\left(I_n-\tfrac{1}{n}\mathbb{1}_n\mathbb{1}_n^\top\right)(f(t))$ denotes the projection of $f(t)$ onto the disagreement submanifold $S^\perp$. \mycolor{This allows us to represent $\bar{\mathscr{L}}_S^2$ in $\bar{\mathscr{L}}^2$ and enables us to treat $\left(\mathscr{L}_S^2, \|\cdot\|_{\mathscr{L}_S^2}\right)$ as if it were a Hilbert space, permitting the definition of passivity.}

\mycolor{For the output agreement problem, the output is expected to converge to the agreement submanifold, while the input can be any signal in $\bar{\mathscr{L}}^2$. With this understanding, we can study the passivity relations \citep[equation (51)]{Max2017submanifold}, i.e., the relations $H\subset \left(\bar{\mathscr{L}}^2, \bar{\mathscr{L}}_S^2\right)$ defined as
%.So, we can define a relation $H$ where $(u(t),y(t))\in H\subset \left(\bar{\mathscr{L}}^2, \bar{\mathscr{L}}_S^2\right)$. Then the passivity relation for $H$ is described by, 
%
\begin{equation}\label{eq:passivity_relations}
\begin{aligned}
H = \Bigg\{ \, (u(t), y(t)) \,\Big|\, 
    & \, u(t) \in \bar{\mathscr{L}}^2,\ y(t) \in \bar{\mathscr{L}}_S^2,\ \forall \tau \in [0,\infty), \\
    &\hspace{-2.5cm} \left\langle u^\tau(t), \Theta_S\big(y_S^\tau(t)\big)\right\rangle 
    \,\geq\, l\,\|u^\tau(t)\|_{\mathscr{L}^2}^2 
    \,+\, e\,\big\|\Theta_S\big(y_S^\tau(t)\big)\big\|_{\mathscr{L}^2}^2 
\Bigg\}.
\end{aligned}
\end{equation}
% \begin{equation}\label{eq:passivity_relations}
%     \begin{split}
%         H = \left\{\right.   (u(t),y(t))\ | \, u(t) \in \bar{\mathscr{L}}^2, y(t) \in \bar{\mathscr{L}}_S^2, \ \forall \tau\in [0,\infty), \right.\\ &  
%  & \left. \left\langle u^\tau(t), \Theta_S(y_S^\tau(t))\right\rangle \geq l\|u^\tau(t)\|_{\mathscr{L}^2}^2+e\|\Theta_S(y_S^\tau(t))\|_{\mathscr{L}^2}^2  \right\} ,
%     \end{split}
% \end{equation}
where $u^\tau(t)$ and $y_S^\tau(t)$ are as defined in \eqref{eq:f_S_tau}. %we use the fact that $\bar{\mathscr{L}}^2=\bar{\mathscr{L}}_{\{0\}}^2$ \citep{Max2017submanifold}, and 
The relation \eqref{eq:passivity_relations} is said to be passive when $l=e=0$. In general, $l$ and $e$ can be any real numbers. Then, applying \eqref{eq:ThetaS} and the properties of $\Theta_M$ (see \citep[Lemma 1]{Max2017submanifold}), we can write down the point-wise form of the passivity inequality,
\begin{equation}
    u(t)^\top \proj_{S^\perp}(y(t))\geq l\|u(t)\|^2+e\|\proj_{S^\perp}(y(t))\|^2,
\end{equation}
where $\proj_{S^\perp}(y(t))$ appears explicitly. This connects the passivity theory to the output consensus problem.}

Given these notions, demonstrating convergence to the submanifold $S$ is equivalent to proving that $\proj_{S^\perp}(y(t))$ approaches zero as $t\to \infty$. This equivalence arises from the geometric interpretation of the projection operator: as $\proj_{S^\perp}(y(t))$ tends to zero, the distance between the signal $y(t)$ and its projection onto the submanifold $S$ diminishes, implying convergence to $S$. The following definition connects output consensus and convergence to a submanifold.
\begin{definition}
    Consider a network system consisting of a group of agents and edge controllers interconnected as in Figure \ref{fig:DC_directed}. Let $y(t)$ be the output of the system. We say that output $y(t)$ \emph{asymptotically converges to the agreement submanifold $S$}, if
    % \begin{equation}
        $\lim_{t\to\infty}\proj_{S^\perp}(y(t))= 0.$
    % \end{equation}
\end{definition}
For conciseness, in the following discussion, we adopt the notation $\proj_{S^\perp}(y)$ in place of $\proj_{S^\perp}(y(t))$.

\section{A passivity analysis for directed coupling}\label{sec:general_structure}
This section begins with a passivity-based analysis of $(\Sigma,\Pi,\mathcal{G})_E$ and $(\Sigma,\Pi,\mathcal{G})_{B_o}$ under the linear consensus protocol, revealing a potential loss of passivity in the feedback path of $(\Sigma,\Pi,\mathcal{G})_{B_o}$. To address this issue, we propose a general approach for analyzing directed information exchange topologies.

\subsection{Passivity analysis for the linear consensus protocol}
The system $(\Sigma,\Pi, \mathcal{G})_{B_o}$ in Figure \ref{fig:DC_directed} might be the most straightforward candidate to analyze directed coupling. However, when applying the basic linear consensus protocol for digraphs to this structure, the passivity of the feedback path (from $y$ to $g$) cannot be guaranteed even though the edge controllers are output-strictly passive. We focus our analysis on $(\Sigma,\Pi, \mathcal{G})_{B_o}$, as the approach and results for the alternative case are analogous.

Consider the linear consensus protocol for digraphs. Here, we take the agent dynamics $\Sigma^l$ to be the integrators, and the controller dynamics $\Pi^l$ to be the linear static map,
\begin{equation} \label{eq:integrators}
   \Sigma^l: \begin{cases}
       \dot{x}(t) = u(t),\\ y(t) = x(t),
   \end{cases} \text{and} \ \ \Pi^l:  \mu(t) = \zeta(t).
\end{equation}
Note that the integrator dynamics are passive \citep{Khalil2008Nonlinear}, and edge controllers are output strictly passive.
The closed-loop dynamics then yield
    $\dot{ x}(t)=-L_o(\mathcal{G}){x}(t),$
and the generated trajectories converge to the agreement space, $S={\rm span}(\mathbb{1})$ if and only if the underlying digraph contains \mycolor{a globally reachable node}. 

% Furthermore, if the given digraph, the dynamics (\ref{eq:laplacianflow}) converge to the average agreement set for any initial conditions ${ x}(0)$,$ \lim\limits_{t\rightarrow\infty}{ x}(t)=\tfrac{1}{|\mathbb{V}|}\mathbb{1}\mathbb{1}^T { x}(0)$.

To leverage the benefits of passivity theory for analyzing the diffusively-coupled structure, both the forward and feedback paths should be passive \citep{Khalil2008Nonlinear}. Consider the system $(\Sigma^l,\Pi^l,\mathcal{G})_{B_o}$, where $\Sigma^l$ and $\Pi^l$ are known to be passive. Our objective is to investigate whether the feedback path (from $y$ to $g$) in Figure \ref{fig:DC_directed} is passive. The controllers in this protocol are memoryless functions. Consequently, with input $y$ and output $g$, the feedback path is passive if $y^\top g\geq0$ for all $y$ and $g$ \citep{Khalil2008Nonlinear}. Using the relation $u=-B_o\mu$ and $\mu=\zeta$, it is equivalent to the spectral analysis of the symmetric part of $L_o$, denoted by $\ \ \tfrac {(L_o+L_o^\top)}{2}=\mathrm{sym}(L_o)$ \citep{horn2012matrix}. Indeed, if $y^\top g=y^\top(- B_o)E^\top y = y^\top L_oy = y^\top\tfrac{L_o+L_o^\top}{2}y=y^\top\mathrm{sym}(L_o)y\geq 0$ for all $y\in \mathbb{R}^n$, the feedback path is passive. 

Our first result shows that for digraphs with globally reachable nodes, the smallest eigenvalue of $\mathrm{sym}(L_o)$ is non-positive.
\begin{proposition}
    If $\mathcal{G}$ contains a globally reachable node, then the smallest eigenvalue of the symmetric part $\mathrm{sym}(L_o)$ is non-positive. 
\end{proposition}
\begin{proof}
    Let $s_1, \ldots, s_n$ be the singular values of $L_o$ and $\lambda_1,\ldots,\lambda_n$ be the eigenvalues of $\mathrm{sym}(L_o)$, both arranged in nonincreasing order. 
    The digraph $\mathcal{G}$ containing a globally reachable node implies that the rank of $L_o$ is $n-1$ \citep{bullo2020lectures}. It follows that $s_{n-1}>s_n=0$. To establish the relationship between $s_j$ and $\lambda_j$, we apply the Fan-Hoffman \citep[Proposition III.5.1]{bhatia1997matrix}. This proposition implies that $\lambda_j\leq s_j$ for all $j\in[1,n]$. By setting $j=n$, we can deduce that the smallest eigenvalue of $\mathrm{sym}(L_o)$ is non-positive.
\end{proof}

The following proposition provides a sufficient and necessary condition for $\mathrm{sym}(L_o)$ having a zero eigenvalue.
\begin{proposition}
    Let $\mathcal{G}$ contains a globally reachable node. Then $L_o$ and $L_o^\top$ have the same kernel space if and only if $\mathrm{sym}(L_o)$ has a zero eigenvalue.
\end{proposition}
\begin{proof}
    We first show the sufficiency. Let $q\neq \mathbb{0}$ be a vector in $S$. Then we have $\tfrac{1}{2}L_oq+\tfrac{1}{2}L_o^\top q=\mathbb{0}=\tfrac{1}{2}(L_o+L_o^\top )q =\mathrm{sym}(L_o)q$ and $(q,0)$ is an eigenpair of $\mathrm{sym}(L_o)$.
   To prove the necessity, let $v\neq \mathbb{0}$ be the eigenvector w.r.t. the $0$ eigenvalue, i.e., $\mathrm{sym}(L_o)v=\mathbb{0}$. It follows that $v^\top \mathrm{sym}(L_o)v=\tfrac{1}{2}v^\top (L_o+L_o^\top)v=\mathbb{0}$. Since $v^\top L_ov=v^\top L_o^\top v$, we have $v^\top \mathrm{sym}(L_o)v=v^\top L_ov=v^\top L_o^\top v=\mathbb{0}$. $v\neq \mathbb{0}$, so the above equalities are satisfied only when $v\in\ker(L_o)$ and $v\in\ker(L_o^\top)$. The existence of a globally reachable node implies that the dimensions of the kernel space of $L_o$ and $L_o^\top$ are $1$, so $\ker(L_o^\top)=\ker(L_o)=S$.
\end{proof}

This proposition suggests that for a general digraph where $L_o$ and $L_o^\top$ don't have the same kernel space, the smallest eigenvalue of $\mathrm{sym}(L_o)$ is negative. Thus, the feedback path may lose passivity, even though the edge controllers are output strictly passive. The following proposition establishes the equivalence between $L_o$ and $L_o^\top$ having the same kernel space and the digraphs being balanced.
\begin{proposition}
    Let $\mathcal{G}$ contain a globally reachable node. Then, $L_o$ and $L_o^\top$ have the same kernel space if and only if $\mathcal{G}$ is balanced.
\end{proposition}
\begin{proof}    
     Recall Lemma \ref{lemma:gmatrix} that a digraph is balanced if and only if $L_o\mathbb{1}=L_o^\top\mathbb{1}=\mathbb{0}$. Now, suppose that $L_o$ and $L_o^\top$ have the same kernel space. According to \citep[Lemma 6.2]{bullo2020lectures} and \citep[Theorem 6.6]{bullo2020lectures}, the kernel space of $L_o^\top$ is given by $S={\rm span}(\mathbb{1})$, implying $\mathbb{1}\in\ker(L_o)$ and $\mathbb{1}\in\ker(L_o^\top)$.

    Conversely, suppose that $\mathcal{G}$ is balanced. The given conditions imply that both $L_o$ and $L_o^\top$ have one-dimensional kernel spaces, with $\mathbb{1}$ in both $\ker(L_o)$ and $\ker(L_o^\top)$. Consequently, we conclude that $\ker(L_o)=\ker(L_o^\top)=S$.
\end{proof}

The above results demonstrate that under linear consensus protocol, only the systems on some specified digraphs can preserve passivity. Moreover, the above passivity analysis only considers the case where the edge controllers follow the simplest dynamics. The passivity analysis may be more tricky if the edge controllers are modeled by more complex dynamics. This suggests we need a more general approach for analyzing MASs on digraphs. 

Note that the feedback path of $(\Sigma^l,\Pi^l,\mathcal{G})_E$ preserves the passivity of the controllers, mirroring the behavior observed for balanced digraphs. This serves as another example showing the intermediate position of balanced digraphs between undirected graphs and unbalanced digraphs.

\subsection{A general approach for directed coupling}
Recall that the passivity of a system is preserved after being post-multiplied by a matrix and pre-multiplied by its transpose \citep{arcak2007passivitydesign}. Also, the incidence matrix can be represented as $E=B_i+B_o$. Equivalently, the incidence matrix for a direcrted graph can be expressed as $B_o = E - B_i$.  Inspired by these, we use the decomposition idea to design a structure capable of conducting passivity analysis for MASs over digraphs, as illustrated in Figure \ref{fig:decomp}.

To derive the structure depicted in Figure \ref{fig:decomp}, let us begin by examining $(\Sigma,\Pi,\mathcal{G})_{B_o}$. By viewing Figure \ref{fig:DC_directed} from a different perspective, where $u=-B_o\mu=-(E-B_i)\mu$, we can decompose its feedback loop into two distinct branches. The first branch transmits the signals for $y$ to $z$, and the second branch transmits the signals for $y$ to $w$. The first branch and the forward path form a feedback connection, as highlighted in the gray box in Figure \ref{fig:decomp}. This sub-structure is  the diffusively-coupled network $(\Sigma,\Pi,\mathcal{G})_E$, and the feedback connection is passive, provided that the agents and edge controllers are passive. Indeed, the first branch preserves the passivity of the edge controllers, and the inner product $z^\top y$ satisfies $$z^\top y=\mu^\top E^\top y=\mu^\top\zeta\geq \dot{V},$$ 
where $V(\eta)$ denotes a continuously differentiable positive semidefinite function known as the storage function.  Consequently, by Theorem 6.1 in \citep{Khalil2008Nonlinear}, the feedback connection is passive for all input-output pairs.

With this understanding, we can treat $w$ as an external input that carries directed information to the inner-feedback loop (outlined in grey in Figure \ref{fig:decomp}). Although the passivity of the overall system $(\Sigma, \Pi,\mathcal G)_{B_o}$ cannot be guaranteed, we can still exploit the passivity properties preserved in the inner-feedback loop and perform analysis on the feedback interconnection with the input-output pair $(w,y)$.

To differentiate between the structures represented in Figure \ref{fig:DC_directed} and Figure \ref{fig:decomp}, we introduce the notation $(\Sigma,\Pi,\mathcal{G},w)$ to denote the system depicted in Figure \ref{fig:decomp}. Now, we can define new diffusively-coupled relations for $(\Sigma,\Pi,\mathcal{G},w)$. Let $w(t)=B_i\mu(t)$ and $z(t)=E\mu(t)$. It follows that,
\begin{align}
    u(t)&=w(t)-z(t)=-B_o\mu(t) \label{eq:u_Bomu}\\
    \zeta(t)&=E^\top y(t),
\end{align}
where the structure reduces to $(\Sigma,\Pi,\mathcal{G})_E$ when $w(t)=0$.
This decomposition serves as a general approach to handling directed coupling and allows us to analyze the system's behavior using passivity theory. 

Recall that the output consensus problem can be transformed to an equivalent problem of analyzing convergence to a submanifold. To apply the idea of passivity relations \eqref{eq:passivity_relations}, define two relations, i.e., the agent relation $H_a$, 
$$(u(t),y(t))\in H_a\subset (\bar{\mathscr{L}}^2,\bar{\mathscr{L}}_S^2),$$ and the controller relation $H_c$,
$$(y(t),z(t))\in H_c\subset (\bar{\mathscr{L}}_S^2,\bar{\mathscr{L}}^2).$$ We can now define the problem that we will consider.
\begin{problem}\label{pb:outagreement}
    Consider the network system $(\Sigma,\Pi,\mathcal{G},w)$. Under what passivity conditions on $H_a$ and $H_c$ does the output of the system converge to the agreement submanifold?
\end{problem}

\section{\mycolor{Stabilization} of network systems over balanced digraphs}\label{sec:main_result}
This section focuses on a particular type of digraph, denoted by $\mathcal{G}_b$, which is characterized by being balanced and having a globally reachable node. \mycolor{We consider the stabilization problem, which is a specific form of Problem \ref{pb:outagreement}, where the objective is to ensure that the outputs of all agents converge to zero, i.e., $\lim\limits_{t\to\infty}y(t)=\mathbb{0}\in\mathrm{span}(\mathbb{1})$. Addressing this problem provides a foundation for solving the more general output consensus problem for networked systems over general digraphs (i.e., Problem \ref{pb:outagreement}).}

Assume that for $i\in\mathbb{V}$, the agents follow the dynamics,
\begin{equation}\label{eq:sigma_o}
    \begin{aligned}
        \Sigma_i^o:\begin{cases}
            \dot{x}_i(t)=f_i(x_i(t),u_i(t)),\\
        y_i(t)=h_i(x_i(t)),
        \end{cases}
    \end{aligned}
\end{equation}
where $f_i$ and $h_i$ are continuously differentiable functions. \mycolor{For the system \eqref{eq:sigma_o}, we say $x_0$ is \emph{asymptotically reachable from $\{0\}$} if there exists an input $u_{x_0}:(-\infty,0]\to\mathbb{R}^{\vert\mathbb{V}\vert}$ such that, when applying $u_{x_0}$ to the state equation of \eqref{eq:sigma_o} for $t\in(-\infty,0]$, we have $x(0)=x_0$ and $x(t)\to 0$ as $t\to -\infty$. }

\mycolor{Consider the output agreement problem of a network system $(\Sigma^o,\Pi,\mathcal{G}_b,w)$. We first establish the passivity-like inequalities of $H_a$ and $H_c$ by exploiting the inherent passivity properties of the individual agents and controllers. Then, we derive a sufficient condition to guarantee output agreement of the system.}

The following result provides a passivity-like inequality for the agent relation $H_a$.
\begin{proposition}\label{prop:agents_passive}
   Consider a group of $|\mathbb{V}|$ SISO agents \eqref{eq:sigma_o} \mycolor{interconnected over a digraph $\mathcal{G}_b$}. Assume that each agent $\Sigma_i^o$, for $i\in\{1,\ldots,|\mathbb{V}|\}$, is OP-$\varepsilon_i$ \mycolor{and with initial conditions that are asymptotically reachable from $\{0\}$}. Let $\varepsilon=\min_i(\varepsilon_i)$. Then, it follows that 
    \begin{equation*}\label{eq:agentpp}
        u^\top \proj_{S^\perp}(y)\geq \sum_{i=1}^{|\mathbb{V}|} \dot{Q}_i(x_i)-\|u\|_2\|y\|_2+\varepsilon \|\proj_{S^\perp}(y)\|_2^2,
    \end{equation*}
    \mycolor{and the passivity relation satisfies,
    \begin{equation}\label{eq:int_agent_relation}
        \langle u^\tau, \proj_{S^\perp}(y^\tau) \rangle \hspace{-3pt} \geq \hspace{-3pt} -\tfrac{\max(D_o)}{\varepsilon}\|\mu^\tau\|_{\mathscr{L}_2}^2 \hspace{-3pt}+\varepsilon\|\proj_{S^\perp}(y^\tau)\|_{\mathscr{L}_2}^2,
    \end{equation}}
    where $Q_i(x_i)$ is the storage function, and \mycolor{$\max(D_o)$} denotes \mycolor{the maximal out-degree of $\mathcal{G}_b$}, respectively.
\end{proposition}
\begin{proof}
    We start by summing up the passivity inequalities of all the agents, i.e., 
    \begin{equation}\label{eq:op_agents}
        u^\top y\geq\sum_{i=1}^{|\mathbb{V}|} \dot{Q}_i+\sum_{i=1}^{|\mathbb{V}|} \varepsilon_iy_i^2\geq\sum_{i=1}^{|\mathbb{V}|} \dot{Q}_i +\varepsilon \|y\|_2^2.
    \end{equation}
    Then, consider $u^\top\proj_{S^\perp}(y)$,
    \begin{equation}\label{eq:proj_agents}
        \begin{aligned}
            u^\top\proj_{S^\perp}(y)&=u^\top y-\tfrac{1}{|\mathbb{V}|}u^\top\mathbb{1}\mathbb{1}^\top y \\
            &\geq\sum_{i=1}^{{|\mathbb{V}|}} \dot{Q}_i +\varepsilon \|y\|_2^2-\tfrac{1}{|\mathbb{V}|}u^\top\mathbb{1}\mathbb{1}^\top y.
        \end{aligned}
    \end{equation}
For $-\tfrac{1}{|\mathbb{V}|}u^\top\mathbb{1}\mathbb{1}^\top y$, by the Cauchy-Schwarz inequality,
\begin{align*}
        |-\tfrac{1}{|\mathbb{V}|}u^\top\mathbb{1}\mathbb{1}^\top y|&=\tfrac{1}{|\mathbb{V}|}|u^\top\mathbb{1}\mathbb{1}^\top y|=\tfrac{1}{|\mathbb{V}|}|u^\top\mathbb{1}||\mathbb{1}^\top y| \\
        &\leq \tfrac{1}{|\mathbb{V}|} \|u\|_2 \|\mathbb{1}\|_2\|y\|_2 \|\mathbb{1}\|_2=\|u\|_2\|y\|_2.
\end{align*}
So, we arrive at 
\begin{equation}\label{eq:bound_uy}
    -\tfrac{1}{|\mathbb{V}|}u^\top\mathbb{1}\mathbb{1}^\top y\geq -\|u\|_2\|y\|_2.
\end{equation} 

Before considering $\|\proj_{S^\perp}(y)\|_2^2$, recall that $I-\tfrac{1}{|\mathbb{V}|}\mathbb{1}\mathbb{1}^\top$ is a projection matrix with eigenvalues $\{0,1^{(|\mathbb{V}|-1)}\}$. Using the properties of the Rayleigh quotient \citep[Theorem 4.2.2]{horn2012matrix}, we get 
\begin{equation}\label{eq:proj_y}
    \|\proj_{S^\perp}(y)\|_2^2=y^\top(I-\tfrac{1}{|\mathbb{V}|}\mathbb{1}\mathbb{1}^\top)y\leq y^\top y.
\end{equation}

Pluging \eqref{eq:proj_y} and \eqref{eq:bound_uy} into \eqref{eq:proj_agents}, we obtain \eqref{eq:agentpp}.

\mycolor{To show \eqref{eq:int_agent_relation}, we first rearrange \eqref{eq:op_agents},
    \begin{align*}
        \sum_{i=i}^{\vert \mathbb{V}\vert}\dot{Q}_i &\leq u^\top y-\varepsilon y^\top y=-\tfrac{1}{2\varepsilon}(u-\varepsilon y)^\top (u-\varepsilon y)\\
        &+\tfrac{1}{2\varepsilon} u^\top u-\tfrac{\varepsilon}{2}y^\top y \leq \tfrac{1}{2\varepsilon}u^Tu-\tfrac{\varepsilon}{2}y^\top y.
    \end{align*}
    Integrating both sides over $(-\infty,\tau]$ yields,
    \begin{align*}
        \int_{-\infty}^\tau y^\top(t) y(t) {\rm d}t&\leq\tfrac{1}{\varepsilon^2} \int_{-\infty}^\tau u^\top(t) u(t) {\rm d}t-\tfrac{2}{\varepsilon}\sum_{i=i}^{\vert \mathbb{V}\vert}Q_i(x(\tau))\\
        &+\tfrac{2}{\varepsilon}\lim_{t\to \infty}\sum_{i=i}^{\vert \mathbb{V}\vert}Q_i(x(-t)).
    \end{align*}
    Employing $\lim_{t\to \infty}\sum_{i=i}^{\vert \mathbb{V}\vert}Q_i(x(-t))=0$ (i.e., the asymptotical reachability of initial conditions) and the non-negative nature of $Q(x)$, it follows that,}
    \begin{equation}\label{eq:L2_agents}
        \|y^\tau\|_{\mycolor{\mathscr{L}_2}}\leq \tfrac{1}{\varepsilon}\|u^\tau\|_{\mycolor{\mathscr{L}_2}}.
    \end{equation}
    Now, by integrating $-\|u\|_2\|y\|_2$ over $(-\infty,\tau]$ and using Cauchy-Schwarz, we have that 
    \begin{equation}\label{eq:int_uy}
        \left\vert-\int_{-\infty}^\tau\|u(t)\|_2\|y(t)\|_2 {\rm d}t \right\vert  \leq \|u^\tau\|_{\mycolor{\mathscr{L}_2}}\|y^\tau\|_{\mycolor{\mathscr{L}_2}}.
    \end{equation}
    With inequalities \eqref{eq:L2_agents} and \eqref{eq:int_uy}, we obtain that
    \begin{equation}\label{eq:op_uandy}
        -\int_{-\infty}^\tau\|u(t)\|_2\|y(t)\|_2 {\rm d}t\geq -\|u^\tau\|_{\mycolor{\mathscr{L}_2}}\|y^\tau\|_{\mycolor{\mathscr{L}_2}}\geq -\tfrac{1}{\varepsilon}\|u^\tau\|_{\mycolor{\mathscr{L}_2}}^2.
    \end{equation}
    Next, we establish the relationship between $u$ and $\mu$. Since $u=-B_o\mu$ (see \eqref{eq:u_Bomu}), the following relation holds,
    \begin{align}\label{eq:L2_umu}
        -\|u^\tau\|_{\mycolor{\mathscr{L}_2}}^2=-\int_{-\infty}^\tau\mu^\top(t) B_o^\top B_o\mu(t){\rm d}t\geq-\max(D_o)\|\mu^\tau\|_{\mycolor{\mathscr{L}_2}}^2,
    \end{align}
    where the last inequality is by the properties of the Geršgorin Disks Theorem \citep[Theorem 2.8] {bullo2020lectures} and Rayleigh quotient \citep[Theorem 4.2.2]{horn2012matrix}. Indeed, observe that the entries of $B_o^\top B_o$ are either $0$ or $1$, with diagonal elements equal to $1$. Furthermore, the largest row sum of this matrix is given by $\max(D_o)$. This implies that the maximal eigenvalue of $B_o^\top B_o$ is less than or equal to $\max(D_o)$.

    \mycolor{Now, integrate both sides of \eqref{eq:agentpp} over $(-\infty,\tau]$ for $\tau\in[0,\infty)$, 
    \begin{equation}\label{eq:int_agentpp}
    \begin{aligned}
        \langle u^\tau,\proj_{S^\perp}(y^\tau) \rangle &\geq \sum_{i=1}^{\vert\mathbb{V}\vert} Q_i(x_i(\tau))-\lim_{t\to -\infty}\sum_{i=1}^{\vert\mathbb{V}\vert} Q_i(x_i(t)) \\
        &-\int_{-\infty}^\tau
        \|u\|_2\|y\|_2{\rm d}t+\varepsilon\|\proj_{S^\perp}(y^\tau)\|_{\mathscr{L}_2}^2.
        \end{aligned}
    \end{equation}
Recall that $\lim_{t\to -\infty}\sum_{i=1}^{\vert\mathbb{V}\vert} Q_i(x_i(t))=0$ and $Q_i(x_i(t))$ are non-negative. Plugging \eqref{eq:op_uandy} and \eqref{eq:L2_umu} into \eqref{eq:int_agentpp}, we get equation \eqref{eq:int_agent_relation}.}
\end{proof}

The following proposition is for the controller relation $H_c$.
\begin{proposition}\label{prop:controller_passive}
    Consider a group of $|\mathbb{E}|$ SISO edge controllers \eqref{eq:controllers_decomp}.  Assume that for all $k\in\{1,\ldots,|\mathbb{E}|\}$, the controllers $\Pi_k$ are OP-$\alpha_i$. Then, it follows that,
    \begin{equation}\label{eq:controllerpp}
        z^\top \proj_{S^\perp}(y)\geq \sum_{k=1}^{|\mathbb{E}|} \dot{W}_k(\eta_k)+\alpha\|\mu\|_2^2,
    \end{equation} 
    \mycolor{and the passivity relation satisfies,
\begin{equation}\label{eq:int_controller_relation}
    \begin{aligned}
        \langle z^\tau,\proj_{S^\perp}(y^\tau) \rangle\geq %\sum_{k=1}^{|\mathbb{E}|} ({W}_k(\eta_k(\tau))
        -\hspace{-3pt}\lim_{t\to -\infty}\hspace{-2pt}\sum_{k=1}^{|\mathbb{E}|} {W}_k(\eta_k(t))
        +\alpha\|\mu^\tau\|_{\mathscr{L}_2}^2,
    \end{aligned}
    \end{equation}}
    where $W_k(\eta_k)$ are the storage functions and $\alpha=\min\limits_k(\alpha_k)$.
\end{proposition}
\begin{proof}
    Recall that for the incidence matrix of a balanced digraph, the relation $E^\top \mathbb{1}_n=\mathbb{0}_m$ and $E\mathbb{1}_m=\mathbb{0}_n$ hold. Thus,
    \begin{equation*}
    \begin{aligned}
        E^\top y&=(I_m-\tfrac{1}{n}\mathbb{1}_m\mathbb{1}_m^\top)(E^\top y)=\proj_{S^\perp}(E^\top y) \\
        &=E^\top(I_n-\tfrac{1}{n}\mathbb{1}_n\mathbb{1}_n^\top)y=E^\top \proj_{S^\perp}(y).
    \end{aligned}
    \end{equation*}
    With this relation, summing up the passivity inequalities of all the controllers yields,
    \begin{equation}
    \begin{aligned}
         \mu^\top\zeta&=\mu^\top E^\top y=z^\top \proj_{S^\perp}(y)\\
         &\geq \sum_{k=1}^{|\mathbb{E}|} \dot{W}_k+\sum_{k=1}^{|\mathbb{E}|} \alpha_k \mu_k^2 
         \geq \sum_{k=1}^{|\mathbb{E}|} \dot{W}_k+\alpha\|\mu\|_2^2.
    \end{aligned}
    \end{equation}
    This demonstrates \eqref{eq:controllerpp}. \mycolor{Now, integrate both sides of \eqref{eq:controllerpp} over $(-\infty,\tau]$ for $\tau\in[0,\infty)$, and we obtain \eqref{eq:int_controller_relation}.}
\end{proof}

\mycolor{Recall that achieving output agreement suggests \\ $\lim\limits_{t\to\infty}\proj_{S^\perp}(y(t))=0$. Now, we are ready to present the main result of this section.}
\begin{theorem}\label{thm:consensus}
    Consider a network system $(\Sigma^o,\Pi,\mathcal{G}_b,w)$. Suppose the conditions of Proposition \ref{prop:agents_passive} and Proposition \ref{prop:controller_passive} are met.  If $\alpha\geq\tfrac{\max(D_o)}{\varepsilon}$ where $\max(D_o)$ denotes the maximal out-degree of $\mathcal{G}$, then the system is \mycolor{stabilized}.
\end{theorem}

\begin{proof}
    We start by showing that the trajectories of $(\Sigma^o,\Pi,\mathcal{G},w)$ are bounded. From Proposition \ref{prop:agents_passive} and Proposition \ref{prop:controller_passive}, both the forward path (from $u$ to $y$) and the feedback path (from $y$ to $z$) of Figure \ref{fig:decomp} are passive. Consequently, it satisfies a global dissipation inequality, where the rate of change of the storage function is bounded by the supply rate. Since the storage function can be chosen to be radially unbounded (i.e., a quadratic storage function), the trajectories must be bounded.

     Recall that $w=z+u$. Sum up \eqref{eq:int_agent_relation} and \eqref{eq:int_controller_relation}, 
     \begin{small}
\begin{equation}\label{eq:sum_wprojy}
    \begin{aligned}
    \varepsilon\|\proj_{S^\perp}(y^\tau)\|_{\mathscr{L}_2}^2\leq \hspace{-2pt}\langle w^\tau, \proj_{S^\perp}(y^\tau)\rangle+\hspace{-2pt}\lim_{t\to -\infty}\hspace{-3pt}\sum_{i=1}^{|\mathbb{V}|}Q(x_i(t)),
    \end{aligned}
    \end{equation}
    \end{small}
     and the pointwise form yields,
     \begin{small}
         $$w^\top \proj_{S^\perp}(y)\geq \sum_{k=1}^{\vert \mathbb{E}\vert} \dot{W}_k+\varepsilon\|\proj_{S^\perp}(y)\|_2^2.$$
     \end{small}
    Then, following an approach analogous to the proof of Lemma 6.5 in \citep{Khalil2008Nonlinear}, it can be shown that $\proj_{S^\perp}(y)$ is bounded for bounded $w$, as described by,
    \begin{small}
\begin{equation}\label{eq:wy_finite}
\|\proj_{S^\perp}(y^\tau)\|_{\mathscr{L}^2}\leq \tfrac{1}{\varepsilon}\|w^\tau\|_{\mathscr{L}_2}+\hspace{-3pt}\sqrt{\tfrac{2}{\varepsilon}\hspace{-3pt}\lim_{t\to-\infty}\hspace{-2pt}\sum_{k=1}^{|\mathbb{E}|}W(\eta_k(t))}
\end{equation} 
\end{small}

Now, let $s(t)=\proj_{S^\perp}^\top(y(t))\proj_{S^\perp}(y(t))$. From \eqref{eq:sum_wprojy}, due to the boundedness of trajectories and \eqref{eq:wy_finite}, we obtain that \\$\lim_{\tau\to\infty}\int_{0}^\tau s(t){\rm d}t$ exists and is finite.

Next, we demonstrate that $s(t)$ is uniformly continuous. Given the dynamics of $\Sigma_i^o$, the derivative of $y_i$ can be expressed as $\dot{y}_i=\tfrac{\partial h_i}{\partial x_i} f_i(x_i,u_i)$. Since the trajectories are bounded, $\tfrac{\partial h_i}{\partial x_i}$ and $f_i(x_i,u_i)$ are also bounded, implying the boundedness of $\dot{y}_i$ for all $i\in[1,\vert\mathbb{V}\vert]$. Consequently, $\tfrac{\rm d}{{\rm d} y}s(t)=2y^\top(I-\tfrac{1}{|\mathbb{V}|}\mathbb{1}\mathbb{1}^\top)\dot{y}$ is bounded. Thus, $s(t)$ is uniformly continuous.

Now, having satisfied the conditions for applying Barbalat's Lemma~\citep[Lemma 8.2]{Khalil2008Nonlinear} to $s(t)$, we conclude that $\proj_{S^\perp}(y(t)) \to 0$ as $t \to \infty$, which implies asymptotic output agreement. \mycolor{Consequently, $E^\top y(t) = \zeta(t) = \mathbb{0}$. Since both the agents and controllers are output strictly passive, their input-output trajectories pass through the origin, yielding $\zeta(t) = \mu(t) = u(t) = y(t) = \mathbb{0}$. Thus, the network system is stabilized.}
\end{proof}

\begin{figure*}[!ht]
\centering
    \begin{subfigure}{0.3\textwidth}
\centering
    \raisebox{0.8cm}{
    \begin{tikzpicture}[scale = .85]
        % graph
        \Vertex[size=.5,x=0, y=2, label=$1$,opacity =.5]{v1} 
        \Vertex[size=.5,x=0, y=0, label=$2$,opacity =.5]{v2} 
        \Vertex[size=.5,x=-1, y=-1, label=$3$,opacity =.5]{v3}
        \Vertex[size=.5,x=-1, y=1, label=$4$
        ,opacity =.5]{v4}
        \Vertex[size=.5,x=1, y=1, label=$5$,opacity =.5]{v5}
       
        \Edge[Direct](v1)(v4)
        \Edge[Direct](v1)(v5)
        \Edge[Direct](v4)(v2)
        \Edge[Direct](v5)(v2)
        \Edge[Direct](v2)(v1)
        \Edge[Direct](v2)(v3)
        \Edge[Direct](v3)(v1)

       \end{tikzpicture}
       }
    \caption[]{}
    \label{fig:case_graph}
\end{subfigure}
\begin{subfigure}{0.3\textwidth}
    \centering
    \includegraphics[width=\linewidth]{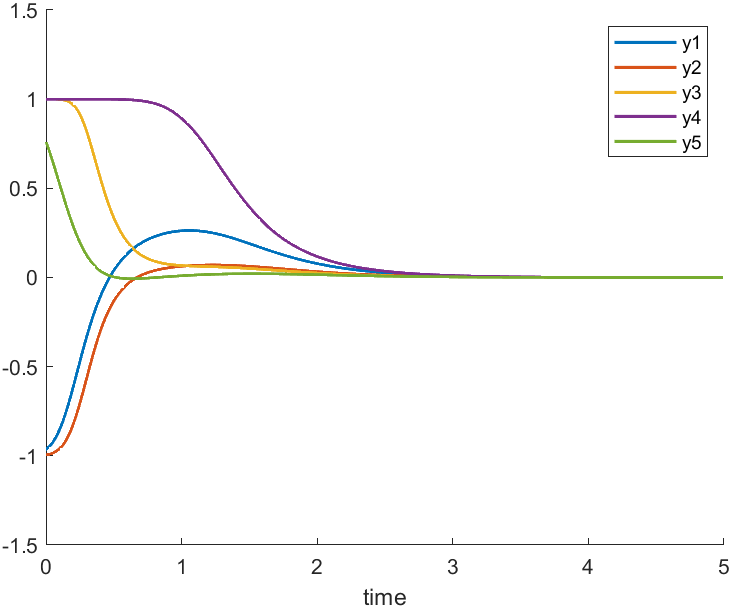}
    \caption[]{}
    \label{fig:traj_linear}
\end{subfigure}
\begin{subfigure}{0.3\textwidth}
    \centering
    \includegraphics[width=\linewidth]{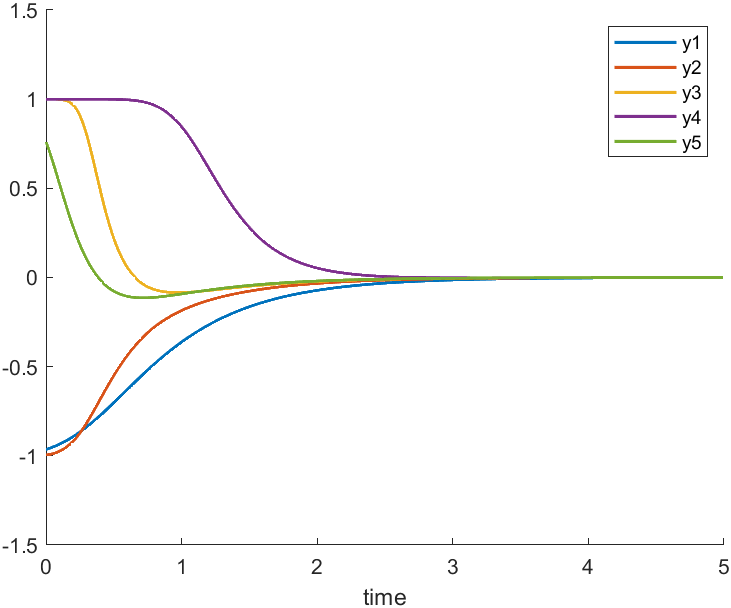}
    \caption[]{}
    \label{fig:diode}
\end{subfigure}
\caption{The underlying graph and trajectories for the outputs. (a) A balanced digraph. (b) Outputs for the system with edge controllers $\mu_k=\tfrac{4}{3}\zeta_k$. (c) Outputs for systems with nonlinear edge controllers $\mu_k=\tfrac{4}{3}\max(\zeta_k,0)$.}
\end{figure*}

Theorem \ref{thm:consensus} presents a passivity-based analysis of network systems with balanced digraphs, establishing a sufficient condition for output consensus (\mycolor{stabilization}) in terms of the passivity index of the edge controller dynamics. The theorem provides a lower bound on the passivity index, which has an insightful physical interpretation: it is the ratio between the maximal out-degree of the underlying digraph and the minimal passivity index of the agents. This result guarantees that the system is \mycolor{stabilized}, but does not necessarily ensure average consensus. Theorem \ref{thm:consensus} also has some limitations that should be acknowledged. First, the provided condition is only sufficient for stabilization, and a tighter lower bound on the passivity index may exist. Second, the theorem is applicable only to balanced digraphs that contain a globally reachable node, and it requires the agents to be output strictly passive with their outputs determined solely by their states.

\section{Case study: Neural Network}\label{sec:case}
In this section, we consider a continuous neural network on $n$ neurons \citep{sharf2018identification,scardovi2009sync_bionet},
\begin{equation}
    \dot{x}_i=-a_ix_i+b\sum_{j\sim i}(\tanh(x_j)-\tanh(x_i))+{\rm w}_i,
\end{equation} where $x_i$ and $\tfrac{1}{a_i}>0$ denote the voltage on the $i$-th neuron and the self-correlation time of the neuron, respectively, $b$ is the coupling coefficient and ${\rm w}_i$ is the exogenous input of the neuron. 
Note that the agents can be modeled by $\dot{x}_i=-a_ix_i+u_i; \ y_i = \tanh(x_i)$ and the edge controllers follow the linear consensus protocol, i.e., $\mu_k=b\zeta_k$.
The agents are OP-$a_i$ and the controllers are OP-$b$.

We run the system with $5$ neurons. The underlying graph of the system, as shown in Figure \ref{fig:case_graph}, is balanced and contains a globally reachable node. The maximum out-degree of the graph is $2$. In this example, the value ${\rm w}_i=0$ and  values $a_i$, are chosen randomly, $a=[1.66, 3.22, 4.62, 1.5, 2.56]$. The initial conditions are set to $x(0)=[-2;-3;6;10;1]$. According to Theorem \ref{thm:consensus}, if $b\geq\tfrac{\max(D_o)}{\min(a_i)}=\tfrac{4}{3}$, the system is stablized. This is verified by the simulation result, as shown in Figure \ref{fig:traj_linear}.

Furthermore, the models of the edge controllers can be chosen to be other nonlinear OP-$\tfrac{4}{3}$ systems, for example, $\mu_k=\tfrac{4}{3}\max(\zeta_k,0)$. Figure \ref{fig:diode} shows the outputs of the system, and the system is stabilized. In both cases, the system states converge to regular consensus.

\section{Concluding remarks}\label{sec:conclusion}
In this work, we first propose a general approach capable of conducting passivity analysis for network systems with directed couplings. Then, we transform the consensus problem into a convergence analysis on a submanifold. Finally, we provide a passivity-based analysis for network systems over balanced digraphs, which serves as a sufficient condition for \mycolor{stabilization (a specific instance of output consensus)}. In future work, we will explore the more general output consensus problem for networked systems over general digraphs.

% \appendix
% \section{Example Appendix Section}


%%       \citept{<label>} ==> Lamport (1994)
%%       \citepp{<label>} ==> (Lamport, 1994)
%%
% Example citation, See \citept{lamport94}.

 \bibliographystyle{elsarticle-harv} 
 \bibliography{ref}

\begin{thebibliography}{20}
\expandafter\ifx\csname natexlab\endcsname\relax\def\natexlab#1{#1}\fi
\providecommand{\url}[1]{\texttt{#1}}
\providecommand{\href}[2]{#2}
\providecommand{\path}[1]{#1}
\providecommand{\DOIprefix}{doi:}
\providecommand{\ArXivprefix}{arXiv:}
\providecommand{\URLprefix}{URL: }
\providecommand{\Pubmedprefix}{pmid:}
\providecommand{\doi}[1]{\href{http://dx.doi.org/#1}{\path{#1}}}
\providecommand{\Pubmed}[1]{\href{pmid:#1}{\path{#1}}}
\providecommand{\bibinfo}[2]{#2}
\ifx\xfnm\relax \def\xfnm[#1]{\unskip,\space#1}\fi
%Type = Article
\bibitem[{Arcak(2007)}]{arcak2007passivitydesign}
\bibinfo{author}{Arcak, M.}, \bibinfo{year}{2007}.
\newblock \bibinfo{title}{Passivity as a design tool for group coordination}.
\newblock \bibinfo{journal}{IEEE Transactions on Automatic Control}
  \bibinfo{volume}{52}, \bibinfo{pages}{1380--1390}.
%Type = Book
\bibitem[{Bai et~al.(2011)Bai, Arcak and Wen}]{bai2011cooperative}
\bibinfo{author}{Bai, H.}, \bibinfo{author}{Arcak, M.}, \bibinfo{author}{Wen,
  J.}, \bibinfo{year}{2011}.
\newblock \bibinfo{title}{Cooperative control design: a systematic,
  passivity-based approach}.
\newblock \bibinfo{publisher}{Springer}.
%Type = Book
\bibitem[{Bhatia(1997)}]{bhatia1997matrix}
\bibinfo{author}{Bhatia, R.}, \bibinfo{year}{1997}.
\newblock \bibinfo{title}{Matrix analysis}. volume \bibinfo{volume}{169}.
\newblock \bibinfo{publisher}{Springer Science \& Business Media}.
%Type = Book
\bibitem[{Bullo(2020)}]{bullo2020lectures}
\bibinfo{author}{Bullo, F.}, \bibinfo{year}{2020}.
\newblock \bibinfo{title}{Lectures on network systems}.
\newblock \bibinfo{publisher}{Kindle Direct Publishing Seattle, DC, USA}.
%Type = Article
\bibitem[{B{\"u}rger et~al.(2013)B{\"u}rger, Zelazo and
  Allg{\"o}wer}]{Burger2013AUTO_duality}
\bibinfo{author}{B{\"u}rger, M.}, \bibinfo{author}{Zelazo, D.},
  \bibinfo{author}{Allg{\"o}wer, F.}, \bibinfo{year}{2013}.
\newblock \bibinfo{title}{Duality and network theory in passivity-based
  cooperative control}.
\newblock \bibinfo{journal}{Automatica} \bibinfo{volume}{50},
  \bibinfo{pages}{2051--2061}.
%Type = Article
\bibitem[{Chung et~al.(2018)Chung, Paranjape, Dames, Shen and
  Kumar}]{AerialSwarm_TRO2018}
\bibinfo{author}{Chung, S.J.}, \bibinfo{author}{Paranjape, A.A.},
  \bibinfo{author}{Dames, P.M.}, \bibinfo{author}{Shen, S.},
  \bibinfo{author}{Kumar, V.R.}, \bibinfo{year}{2018}.
\newblock \bibinfo{title}{A survey on aerial swarm robotics}.
\newblock \bibinfo{journal}{IEEE Transactions on Robotics}
  \bibinfo{volume}{34}, \bibinfo{pages}{837--855}.
%Type = Book
\bibitem[{Horn and Johnson(2012)}]{horn2012matrix}
\bibinfo{author}{Horn, R.A.}, \bibinfo{author}{Johnson, C.R.},
  \bibinfo{year}{2012}.
\newblock \bibinfo{title}{Matrix analysis}.
\newblock \bibinfo{publisher}{Cambridge university press}.
%Type = Book
\bibitem[{Khalil(2002)}]{Khalil2008Nonlinear}
\bibinfo{author}{Khalil, H.K.}, \bibinfo{year}{2002}.
\newblock \bibinfo{title}{Nonlinear Systems, 3rd ed}.
\newblock \bibinfo{publisher}{Prentice Hall}.
%Type = Book
\bibitem[{Lee(2003)}]{lee2024smoothmanifolds}
\bibinfo{author}{Lee, J.M.}, \bibinfo{year}{2003}.
\newblock \bibinfo{title}{Introduction to Smooth Manifolds}.
\newblock \bibinfo{publisher}{Springer}.
%Type = Article
\bibitem[{Li et~al.(2020)Li, Chesi and Hong}]{li2020jointlypassivation}
\bibinfo{author}{Li, M.}, \bibinfo{author}{Chesi, G.}, \bibinfo{author}{Hong,
  Y.}, \bibinfo{year}{2020}.
\newblock \bibinfo{title}{Input-feedforward-passivity-based distributed
  optimization over jointly connected balanced digraphs}.
\newblock \bibinfo{journal}{IEEE Transactions on Automatic Control}
  \bibinfo{volume}{66}, \bibinfo{pages}{4117--4131}.
%Type = Article
\bibitem[{Li et~al.(2019)Li, Su and Chesi}]{li2019passivition}
\bibinfo{author}{Li, M.}, \bibinfo{author}{Su, L.}, \bibinfo{author}{Chesi,
  G.}, \bibinfo{year}{2019}.
\newblock \bibinfo{title}{Consensus of heterogeneous multi-agent systems with
  diffusive couplings via passivity indices}.
\newblock \bibinfo{journal}{IEEE Control Systems Letters} \bibinfo{volume}{3},
  \bibinfo{pages}{434--439}.
%Type = Book
\bibitem[{Mesbahi and Egerstedt(2010)}]{mesbahi2010graph}
\bibinfo{author}{Mesbahi, M.}, \bibinfo{author}{Egerstedt, M.},
  \bibinfo{year}{2010}.
\newblock \bibinfo{title}{Graph theoretic methods in multiagent networks}.
\newblock \bibinfo{publisher}{Princeton University Press}.
%Type = Article
\bibitem[{Montenbruck et~al.(2017)Montenbruck, Arcak and
  Allgöwer}]{Max2017submanifold}
\bibinfo{author}{Montenbruck, J.M.}, \bibinfo{author}{Arcak, M.},
  \bibinfo{author}{Allgöwer, F.}, \bibinfo{year}{2017}.
\newblock \bibinfo{title}{An input–output framework for submanifold
  stabilization}.
\newblock \bibinfo{journal}{IEEE Transactions on Automatic Control}
  \bibinfo{volume}{62}, \bibinfo{pages}{5170--5184}.
\newblock \DOIprefix\doi{10.1109/TAC.2017.2679480}.
%Type = Article
\bibitem[{Restrepo et~al.(2021)Restrepo, Lor{\'\i}a, Sarras and
  Marzat}]{restrepo2021edgeLyapunov}
\bibinfo{author}{Restrepo, E.}, \bibinfo{author}{Lor{\'\i}a, A.},
  \bibinfo{author}{Sarras, I.}, \bibinfo{author}{Marzat, J.},
  \bibinfo{year}{2021}.
\newblock \bibinfo{title}{Edge-based strict lyapunov functions for consensus
  with connectivity preservation over directed graphs}.
\newblock \bibinfo{journal}{Automatica} \bibinfo{volume}{132},
  \bibinfo{pages}{109812}.
%Type = Book
\bibitem[{Rockafellar(1998)}]{Rockafellar1998NetOpt}
\bibinfo{author}{Rockafellar, R.T.}, \bibinfo{year}{1998}.
\newblock \bibinfo{title}{{Network Flows and Monotropic Optimization}}.
\newblock \bibinfo{publisher}{Athena Scientific}.
%Type = Inproceedings
\bibitem[{Scardovi et~al.(2009)Scardovi, Arcak and
  Sontag}]{scardovi2009sync_bionet}
\bibinfo{author}{Scardovi, L.}, \bibinfo{author}{Arcak, M.},
  \bibinfo{author}{Sontag, E.D.}, \bibinfo{year}{2009}.
\newblock \bibinfo{title}{Synchronization of interconnected systems with an
  input-output approach. part ii: State-space result and application to
  biochemical networks}, in: \bibinfo{booktitle}{Proceedings of the 48h
  Conference on Decision and Control held jointly with 28th Chinese Control
  Conference}, \bibinfo{organization}{IEEE}. pp. \bibinfo{pages}{615--620}.
%Type = Inproceedings
\bibitem[{Sharf and Zelazo(2018)}]{sharf2018identification}
\bibinfo{author}{Sharf, M.}, \bibinfo{author}{Zelazo, D.},
  \bibinfo{year}{2018}.
\newblock \bibinfo{title}{Network identification: A passivity and network
  optimization approach}, in: \bibinfo{booktitle}{Conference on Decision and
  Control}, \bibinfo{organization}{IEEE}. pp. \bibinfo{pages}{2107--2113}.
%Type = Article
\bibitem[{Sharf and Zelazo(2019)}]{Sharf2019MIMO}
\bibinfo{author}{Sharf, M.}, \bibinfo{author}{Zelazo, D.},
  \bibinfo{year}{2019}.
\newblock \bibinfo{title}{Analysis and synthesis of {MIMO} multi-agent systems
  using network optimization}.
\newblock \bibinfo{journal}{IEEE Transactions on Automatic Control}
  \bibinfo{volume}{64}, \bibinfo{pages}{4512--4524}.
%Type = Article
\bibitem[{Zhang et~al.(2021)Zhang, Yang and
  Ba{\c{s}}ar}]{ReinforcementLearning}
\bibinfo{author}{Zhang, K.}, \bibinfo{author}{Yang, Z.},
  \bibinfo{author}{Ba{\c{s}}ar, T.}, \bibinfo{year}{2021}.
\newblock \bibinfo{title}{Multi-agent reinforcement learning: A selective
  overview of theories and algorithms}.
\newblock \bibinfo{journal}{Handbook of reinforcement learning and control} ,
  \bibinfo{pages}{321--384}.
%Type = Article
\bibitem[{Zhao and Zelazo(2015)}]{BearingOnly_TAC2015}
\bibinfo{author}{Zhao, S.}, \bibinfo{author}{Zelazo, D.}, \bibinfo{year}{2015}.
\newblock \bibinfo{title}{Bearing rigidity and almost global bearing-only
  formation stabilization}.
\newblock \bibinfo{journal}{IEEE Transactions on Automatic Control}
  \bibinfo{volume}{61}, \bibinfo{pages}{1255--1268}.

\end{thebibliography}




% %% For authoryear reference style

% \bibitem[Lamport(1994)]{lamport94}
%   \textit{\LaTeX: a document preparation system},
%   Addison Wesley, Massachusetts,

% \end{thebibliography}
\end{document}